\documentclass[a4paper, 11pt]{amsart}

\usepackage{epsfig}
\usepackage{amsthm,amsfonts}
\usepackage{amssymb,graphicx,color}
\usepackage[all]{xy}
\usepackage{verbatim}
\usepackage{hyperref}
\usepackage{enumitem}
\usepackage{float}
\graphicspath{ {Images/} }

\newtheorem{theorem}{Theorem}[section]
\newtheorem*{theorem*}{Theorem}

\newtheorem{corollary}[theorem]{Corollary}
\newtheorem{proposition}[theorem]{Proposition}
\newtheorem{claim}{Claim}[theorem]
\newtheorem{innercustomthm}{Theorem}
\newenvironment{customthm}[1]
  {\renewcommand\theinnercustomthm{#1}\innercustomthm}
  {\endinnercustomthm}

\theoremstyle{definition}
\newtheorem{definition}[theorem]{Definition}

\newtheorem*{conjecture*}{Conjecture}
\newtheorem{example}[theorem]{Example}

\theoremstyle{remark}
\newtheorem{remark}[theorem]{Remark}

\newcommand{\R}{\mathbb{R}}

\begin{document}

\title[On the link of Lipschitz normally embedded sets]
{On the link of Lipschitz normally embedded sets}

\author[R. Mendes]{Rodrigo Mendes}
\address{Rodrigo Mendes: (1) Instituto de Ci\^encias Exatas e da Natureza, Universidade de Integra\c{c}\~ao Internacional da Lusofonia Afro-Brasileira (Unilab), Campus dos Palmares, 62785-000, Acarape/CE, Brazil
\newline
(2) Departament of Mathematics, Ben Gurion University} 

\email{mendespe@post.bgu.ac.il}

\author[J. Edson Sampaio]{Jos\'e Edson Sampaio}
\address{Jos\'e Edson ~Sampaio:  
Departamento de Matem\'atica, Universidade Federal do Cear\'a, Av.
Humberto Monte, s/n, Campus do Pici - Bloco 914, 60455-760
Fortaleza-CE, Brazil} 

 \email{edsonsampaio@mat.ufc.br}
\thanks{The first named author was supported by a post-doctoral fellowship from the Department of Mathematics and the Kreitman School of Advanced Graduate Studies in Ben Gurion University of Negev at Israel}
\thanks{The second named author was partially supported by CNPq-Brazil grant 310438/2021-7, by the Serrapilheira Institute (grant number Serra -- R-2110-39576), by the ERCEA 615655 NMST Consolidator Grant and also by the Basque Government through the BERC 2018-2021 program and Gobierno Vasco Grant IT1094-16, by the Spanish Ministry of Science, Innovation and Universities: BCAM Severo Ochoa accreditation SEV-2017-0718. 
}
\begin{abstract}
A path-connected subanalytic subset in $\mathbb{R}^n$ is naturally equipped with two metrics: the inner and the outer metrics. We say that a subset is {\it Lipschitz normally embedded (LNE)} if these two metrics are equivalent. In this article, we give some criteria for a subanalytic set to be LNE. It is a fundamental question to know if the LNE property is conical, i.e., if it is possible to describe the LNE property of a germ of a subanalytic set in terms of the properties of its link. We answer this question by introducing a new notion called {\it link Lipschitz normally embedding}. We prove that this notion is equivalent to the LNE notion in the case of sets with connected links. 
\end{abstract}

\maketitle

\section{Introduction}

Given a path-connected subanalytic subset $X$ in the Euclidean space $\R^n$, there are two natural metrics on $X$: the outer metric $d(x_1,x_2)=\|x_1-x_2\|$ (i.e., the distance induced by the Euclidean metric on $\R^n$) and
the inner metric $d_X$ given by
$$
d_X(x_1,x_2)={\rm inf}\{length(\gamma); \, \gamma \mbox{ is a rectifiable path in } X \mbox{ connecting } x \mbox{ and } y\}.
$$
We say that a set $X$ is Lipschitz normally embedded (LNE) if there exists $C\geq 1$ such that $d_X(x,y)\leq C\|x-y\|$ for all $x,y\in X$. In this case, we say also that $X$ is $C$-LNE. 

This definition was introduced by L. Birbrair and T. Mostowski \cite{BirbrairM:2000}, where they just call it normally embedded. As already remarked in \cite{NeumannPP:2019}, Lipschitz Normal Embedding is a very active research area with many recent results giving necessary conditions for a set to be LNE in the real and complex setting, e.g., by L. Birbrair, M. Denkowski, A. Fernandes, D. Kerner, F. Misev, W. D. Neumann, J. J. Nu\~no-Ballesteros, H. M. Perdersen, A. Pichon, M. A. S. Ruas, M. Tibar etc (\cite{BirbrairMN-B:2018},  \cite{DenkowskiT:2019}, \cite{FernandesS:2019}, \cite{KernerPR:2018}, \cite{MisevP:2019}, \cite{NeumannPP:2019} and \cite{NeumannPP:2019b}). 
Recent works show that the LNE property appears in several fields of Mathematics, e.g., 
\begin{itemize}
 \item  Algebraic Geometry: 
 \begin{enumerate}
 \item[-] Minimality: It was proved in \cite{NeumannPP:2019b} that among the rational complex surfaces, the LNE surfaces are exactly the minimal surfaces;
  \item[-] Reduced structures: It was proved in \cite{FernandesS:2019} (see also \cite{DenkowskiT:2019}) that LNE analytic sets have reduced tangent cones;
  \item[-] Space of matrices and determinantal varieties: It was proved in \cite{KernerPR:2018} that some algebraic subsets of the space of matrices, which include the space of rectangular/(skew-)symmetric/hermitian matrices of rank equal to a given number and their closures, and the upper triangular matrices with determinant 0 are LNE. Some generalizations for determinantal varieties were also presented;
 \end{enumerate}
  
 \item Differential Geometry: 
 \begin{enumerate}
 \item[-] Bernstein type theorem: It was proved in \cite{FernandesS:2020} that an LNE complex algebraic set which has linear subspace as its tangent cone at infinity must be an affine linear subspace (see also \cite{Sampaio:2022}); 
 \item[-] Tangent cones: It was proved in \cite{FernandesS:2019} that subanalytic LNE sets have LNE tangent cones;
 \end{enumerate}
 \item Topology:
 \begin{enumerate}
 \item[-] Knot Theory: It was proved in \cite{BirbrairMN-B:2018} that, for a large class of locally LNE analytic parametrized surfaces, the knots presented as the links of such surfaces are always trivial (unknotted);
 \item[-] Fundamental group: It was proved in \cite{FernandesS:2021} that if two compact subanalytic sets have the same LNE constant and are close enough with respect to the Hausdorff distance, then their fundamental groups are isomorphic.
 \end{enumerate} 
\end{itemize}

Let us mention that one of the most important tools to study the topology of singularities is the Local Conical Structure Theorem, which says that {\it given a subanalytic set $X\subset \R^n$ and $p\in \overline{X}$, there exists $\epsilon_0$ such that $X\cap \overline{B(p;\epsilon)}$ is homeomorphic to the cone over $X\cap \mathbb{S}^{n-1}(p;\epsilon)$ for all $0<\epsilon\leq \epsilon_0$, where $B(p;\epsilon)$ (resp. $\mathbb{S}^{n-1}(p;\epsilon)$) denotes the open ball (resp. sphere) centered at 0 with radius $\epsilon$.}
However, there is no metric Local Conical Structure Theorem, even in the case of the inner metric, as was proved by L. Birbrair and A. Fernandes in the remarkable paper \cite{BirbrairF:2008}. It is very useful to describe some property of a germ of a subanalytic set in terms of properties of its link and, when we can do that, we say that such a property is a conical property. In general, it is important to Lipschitz Geometry of Singularities to find metric properties which are conical. In particular, a fundamental question is the following:

\vspace{0.5cm}
\noindent {\bf Question 1.} Is LNE property a conical property?

\vspace{0.5cm}

In this article, we answer positively the above question. A set $X \subset \mathbb{R}^n$ is called LNE at $0$ if there exists an open neighbourhood $U\subset \mathbb{R}^n$ of $0$ such that $X\cap U$ is LNE and we say that $X$ is LLNE at $0$ if there exist $C\geq 1$ and $\delta >0$ such that $X\cap \mathbb{S}^{n-1}(0;t)$ is $C$-LNE for all $0<t\leq \delta$. Thus, we prove the following:

\begin{customthm}{\ref*{main_thm_connected}}
Let $X \subset \mathbb{R}^n$ be a closed subanalytic set, $0 \in X$. Assume that $(X\setminus \{0\},0)$ is a connected germ. Then, $X$ is LNE at $0$ if and only if $X$ is LLNE at $0$.
\end{customthm}

We also present a result in the case that $(X\setminus \{0\},0)$ is not necessarily a connected germ (see Corollary \ref{main_thm_non-connected}).

This article is organized as follows. In Section \ref{sec:preliminaries} the main definitions used in this article are presented and two important results are proved: the LLNE property does not depend on the subanalytic norm chosen on $\mathbb{R}^n$ (see Proposition \ref{LLNE_metric_inv}) and L-regular cells are LLNE (see Proposition \ref{L-regular_llne}). In Section \ref{sec:main_results} the main result of this article is proved (see Theorem \ref{main_thm_connected}) and some examples are presented in order to show that the subanalytic hypothesis in this theorem cannot be removed (see Examples \ref{ex:non-LNE} and \ref{ex:non-LLNE}). 


Most parts of the given definitions are general and, when necessary for our purposes, rely on the subanalytic structure. In order to know more about the subanalytic geometry, see, for instance, \cite{Lojasiewicz:1964}, \cite{Gabrielov:1968} and \cite{BierstoneM:2000}.


\bigskip

\noindent{\bf Acknowledgements}. Rodrigo Mendes thanks Ben Gurion University of Negev for its support and stimulating atmosphere during this research work and also the Basque Center of Applied Mathematics (BCAM) for its hospitality and support during the final part of the preparation of this article. We would like to thank L. Birbrair and A. Fernandes for their interest in this research.

\section{Preliminaries}\label{sec:preliminaries}

\subsection{LNE and LLNE properties}
Let $Z\subset\R^n$ be a path-connected subset. Given two points $q,\tilde{q}\in Z$, we define the \emph{inner distance} on $Z$ between $q$ and $\tilde{q}$ by the number $d_Z(q,\tilde{q})$ below:
$$d_Z(q,\tilde{q}):=\inf\{ \mbox{Length}(\gamma) \ | \ \gamma \ \mbox{is an arc on} \ Z \ \mbox{connecting} \ q \ \mbox{to} \ \tilde{q}\}.$$
\begin{definition}
We say that $Z$ is {\bf Lipschitz normally embedded (LNE)}, if there is a constant $C\geq 1$ such that $d_Z(q,\tilde{q})\leq C\|q-\tilde{q}\|$, for all $q,\tilde{q}\in Z$. We say that $Z$ is {\bf Lipschitz normally embedded set at $p$} (shortly LNE at $p$), if there is a neighbourhood $U$ such that $p\in U$ and $Z\cap U$ is an LNE set or, equivalently, that the germ $(Z,p)$ is LNE. In this case, we say also that $X$ is $C$-LNE (resp. $C$-LNE at $p$). 
\end{definition}
\begin{definition}
Let $X\subset\R^n$ be a subset, $p\in \overline X$ and $X_t:=X\cap \mathbb{S}^{n-1}(p;t)$ for all $t>0$. We say that $X$ is {\bf link Lipschitz normally embedded (at $p$)} (shortly LLNE (at $p$)), if there is a constant $C\geq 1$ such that $d_{X_t}\leq C\|\cdot \|$, for all small enough $t>0$. In this case, we say also that $X$ is $C$-LLNE at $p$. 
\end{definition}
More generally, for $\Lambda\subset \R$, a family $\{X_t\}_{t\in \Lambda}$ of subsets of $\R^n$ is called {\bf LNE with uniform constant} if there exists $C\geq 1$ such that $X_t$ is $C$-LNE, for all $t\in \Lambda$. In this case, the family $\{X_t\}_{t\in \Lambda}$ is called also $C$-LNE. We say that $\{X_t\}_{t\in \Lambda}$ is {\bf LNE with uniform constant at $t=0$} if there exist positive constants $\varepsilon$ and $C$ such that $X_t$ is $C$-LNE for all $t\in \Lambda \cap ([-\varepsilon,0)\cup (0,\varepsilon]) $. In this case, the family $\{X_t\}_{t\in \Lambda}$ is called $C$-LNE at $t=0$.

Let $\|\cdot\|_1$ be a subanalytic norm on $\R^n$. Similarly, we say that {\bf $X \subset \mathbb{R}^n$ is LLNE at $p$ w.r.t. $\|\cdot \|_1$} if there is a constant $C>0$ such that $d_{X_t^{\|\cdot \|_1}}\leq C\|\cdot\|$, for all small enough $t>0$, where $X_t^{\|\cdot \|_1}=\{x \in X; \|x -p\|_1=t\}$. 
In fact, the LLNE property does not depend on the subanalytic norm chosen on $\mathbb{R}^n$ as it is shown in the next result.

\begin{proposition}\label{LLNE_metric_inv}
Let $X$ be a subanalytic set, $p \in \overline X$.  Let $\|\cdot \|_1$ and $\|\cdot \|_2$ be subanalytic norms on $\mathbb{R}^n$. Then $X$ is LLNE at $p$ w.r.t. $\|\cdot \|_1$ if and only if $X$ is LLNE at $p$ w.r.t. $\|\cdot \|_2$.
\end{proposition}
\begin{proof}
We assume without loss of generality that $p=0$. Let us consider the map $H\colon\mathbb{R}^n\rightarrow \mathbb{R}^n$ given by 
$$H(x)=\left\{\begin{array}{ll}
            \frac{\|x \|_1}{\|x \|_2}x,& \ x\neq 0\\
            0,&\ x=0.\\
            \end{array}
            \right. 
$$
Since $\|\cdot\|_1$ and $\|\cdot\|_2$ are subanalytic norms, the map $H$ is clearly a subanalytic bi-Lipschitz map. Then, if $X$ is LLNE at $0$  w.r.t. $\|\cdot \|_1$ then $H(X)$ is LLNE at $0$ w.r.t. $\|\cdot \|_2$. It is because $H(\mathbb{S}_{1}(0,r))=\mathbb{S}_{2}(0,r)$, where $\mathbb{S}_{j}(0,r)$ denotes the sphere centred at $0$ with radius $r$ with respect to the norm $\|\cdot \|_j, \ j=1,2$, for all $r>0$. Now, let us consider the map $\tilde H\colon (H(X), \|\cdot \|_2)\rightarrow (X,\|\cdot \|_2)$ given by $\tilde H(y)=Id\circ H^{-1}(y)$, where $Id\colon (X,\|\cdot \|_1)\rightarrow (X,\|\cdot \|_2)$ is the identity map, which is also a bi-Lipschitz homeomorphism. By \cite[Corollary 0.2]{Valette:2007}, there is a bi-Lipschitz homeomorphism $F\colon (H(X), \|\cdot \|_2)\rightarrow (X,\|\cdot \|_2)$ preserving the distance to $0$. Hence, the set $X$ is LLNE at $0$ w.r.t. $\|\cdot \|_2$. The converse of the proof follows by making the same with $H^{-1}$ instead of $H$. 
\end{proof}

\begin{definition}\label{scaling norm}
Given $v =(v_1,...,v_n)\in \mathbb{R}^n_+=\{(x_1,...,x_n);\, x_i>0,\,\, \forall i\}$, we define the subanalytic norm $\|x\|_{max,v}=max\{v_1|x_1|,\ldots,v_n|x_n|\}$. In fact, $\|x\|_{max,v}$ is a semialgebraic norm. In this case, the sphere $\mathbb{S}_{max,v}^{n-1}(p;t)=\{x \in \R^n; \|x -p\|_{max,v}=t\}$ is the union of scaling sections of the form $\{x;|x_i|\leq |x_j|=\frac{t}{v_j}, \forall i\}$ according with the given $v$. 
\end{definition}

\begin{remark}
From Proposition \ref{LLNE_metric_inv}, a subset $X\subset \mathbb{R}^n$ is LLNE at $p$ if the family $\{X\cap \mathbb{S}_{max,v}^{n-1}(p;t)\}_t$ is $C$-LNE at $t=0$, for some $C>0$ and some $v \in \R^n_+$.
\end{remark}

\medskip
\subsection{Inner and outer contacts}
Given two non-negative functions $f$ and $g$, we write $f\lesssim g$ (resp. $f\gtrsim g$) if there exists some
positive constant $C$ such that $f\leq Cg$ (resp. $g\leq Cf$). We also denote $f \simeq g$ if $f\lesssim g$ and $g\lesssim f$. If $f$
and $g$ are germ of functions on $(X, x_0 )$, we write $f\ll  g$ if $\lim\limits_{x\to x_0} [f (x)/g(x)] = 0.$

\begin{definition}[\cite{BirbrairM:2017}]\label{outerinnercontact}
An arc $\gamma$ with the initial point at $x_o$ is a continuous subanalytic map $\gamma\colon [0,\epsilon_o)\rightarrow \R^n$ such that $\gamma(0)=x_o$. When it does not lead to confusion, we use the same notation for an arc and its image in $\R^n$. Unless otherwise specified, we suppose that arcs are parametrized by the
distance to $x_o$, i.e., $\|\gamma(\epsilon)-x_o\|=\epsilon, \ \epsilon \in (0,\epsilon_o)$. Here we denote the image of $\gamma$ by $Im(\gamma)$. 

Let $\gamma_1$ and $\gamma_2$ be two arcs with initial point $x_o$. We can define the {\bf outer contact} of two arcs $\gamma_1, \gamma_2\colon [0,\epsilon_o)\to \R^n$, denoted by $tord(\gamma_1,\gamma_2)$, by the order at $\epsilon=0$ of the function $\|\gamma_1(\epsilon)-\gamma_2(\epsilon)\|$. Given $X\subset \R^n$ be a closed subanalytic set, $x_o \in X$, it is known that there exists a subanalytic distance $\tilde{d}_X$ on $X$ equivalent to its inner distance $d_X$ (see \cite{Kurdyka:1997} and \cite{BirbrairM:2000}, for instance). Then, we can define the {\bf inner contact} of $\gamma_1$ and $\gamma_2$ by the order at $\epsilon=0$ of the function $\tilde{d}_X(\gamma_1(\epsilon),\gamma_2(\epsilon))$ and it is denoted by $tord_X(\gamma_1, \gamma_2)$.    
\end{definition}

\begin{definition}\label{dist.connect.comp}
Let $X\subset \R^n$ be a subset. Let us denote the collection of the path-connected components of $X$ by $\pi_0(X)$. If $X$ is a path-connected set, we define $d_0(X)=1$ and if $X$ is not a path-connected set, we define
$$
d_0(X)=\inf \{d(Y,Z);\, Y,Z\in \pi_0(X) \mbox{ and } Y\not=Z\},
$$
where $d(Y,Z)=\inf \{\|y-z\|;\,y\in Y\mbox{ and } z\in Z\}$.
\end{definition}

\subsection{L-regular cells}
\begin{definition}[L-regular cells]
A subanalytic set $Y\subset \mathbb{R}^n$ is called an L-regular cell (or an L-regular s-cell) with constant $C$ when it is a point or has one of the following forms, for some coordinates of $\mathbb{R}^n$:
\begin{itemize}
 \item [(i)] $Y=\{(x',y) \in \mathbb{R}^{n-1} \times \mathbb{R}; y=f_0(x'); x' \in \tilde Y\}$ or
 \item [(ii)] $Y=\{(x',y) \in \mathbb{R}^{n-1} \times \mathbb{R}; f_1(x') < y < f_2(x'); x' \in \tilde Y\}$,
\end{itemize} 
where
\begin{enumerate}
\item Each $\tilde Y$ is an analytic submanifold (homeomorphic to an open ball) and has one of the previous forms;
\item Each $f_i$ is analytic on $\tilde Y$, subanalytic on the closure of $\tilde Y$ with $f_1(x')<f_2(x')$, for all $x' \in \tilde Y$;
\item There is a constant $C>0$ such that $\|\nabla f_i(x')\| \leq C,$ for all $x' \in \tilde Y$, $i=0,1,2$.
\end{enumerate}
If $Y$ has the form $(i)$ (resp. $(ii)$), we say that $Y$ has a graph type (resp. a band type).
\end{definition}

\begin{remark}
It is known that an L-regular cell is LNE (see \cite{Kurdyka:1997}, p. 180). 
\end{remark}


We are going to show that an L-regular cell is also LLNE.

\begin{proposition}\label{L-regular_llne}
An L-regular cell $Y \subset \mathbb{R}^n$ is LLNE at $p \in \overline{Y}$ whenever its link at $p$, $link_p(Y)$, is connected. In particular, any ($k\geq 2$)-dimensional L-regular cell $Y \subset \mathbb{R}^n$ is LLNE at $p \in \overline{Y}$.
\end{proposition}
In order to avoid excessive repetitions in the next arguments, whenever necessary, we assume that the parameter $t>0$ is taken small enough.
\begin{proof}
We may assume that $p=0$. 
Given $v =(v_1,...,v_n)\in \mathbb{R}^n_+$, we consider $Y_t^{v}{=}Y\cap \mathbb{S}_{max,v}^{n-1}(0,t){=}\bigcup_{j=1}^nY\cap \{x;\|x\|_{max,v}=|v_jx_j|=t\}$, where $\mathbb{S}_{max,v}^{n-1}(0,t)=\{x\in \R^n;\|x\|_{max,v}=t\}$. Then, by Proposition \ref{LLNE_metric_inv}, $Y$ is LLNE at $0$ if and only if the family $\{Y_t^{v}\}_t$ is $C$-LNE at $t=0$, for some constant $C\geq 1$. If $\dim Y=0$ then $Y$ is obviously LLNE at $0$. Thus, we consider $dim(Y)>0$ and assume that any L-regular cell with dimension $<\dim Y$ is LLNE at $0$. Using the immediate fact that if $Z$ is an L-regular cell which is LLNE at $0$ and $f\colon Z\to \R$ is an analytic function with bounded derivative then $Graph(f)$ is LLNE at $0$ as well, we can assume that $Y$ has a band-like type on the L-regular set $\tilde Y$ and $\dim \tilde{Y}=\dim Y-1$. Using the induction hypothesis, we can assume that $\tilde Y$ is LLNE at $0$ and by Proposition \ref{LLNE_metric_inv}, for a given $v \in \R^n_+$, we obtain that the family $\tilde Y_t^v=\tilde Y\cap \mathbb{S}_{max,v}^{n-2}(0,t)=\cup_{j=1}^{n-1} \tilde Y \cap \{x;\|x\|_{max,v}=|v_jx_j|=t\}$ is LNE with uniform constant at $t=0$. 

\medskip

Moreover, since the function $\|\nabla f_i\|\colon \tilde Y\rightarrow \mathbb{R}$ is bounded, for $i=1,2$, we obtain, on these coordinates, that $Y$ is not tangent to the line $L_y=\{(0,\ldots,0,x_n);\, x_n\in \R\}$. So, we can obtain a conical neighbourhood $\mathcal{C}_\eta=\{(\tilde x,x_n); \|\tilde x\|_{max}\leq \eta |x_n|\}$ of $L_y$ such that  $\mathcal{C}_\eta\cap Y=\{0\}$. Thus, we set $v=(1,...,1,\frac{\eta}{2})$, which implies $Y_{t,n}^v=Y\cap \{x;\|x\|_{max,v}=|v_nx_n|=t\}=\emptyset$.

\begin{claim}\label{claim_family_llne}
 The family $\{Y_t^v\}_t$ is LNE with uniform constant at $t=0$. 
\end{claim}
\begin{proof}[Proof of Claim \ref{claim_family_llne}]
We denote $Graph(f_1|_{\tilde Y_t^v}) \cup Graph(f_2|_{\tilde Y_t^v})$ by $\bar\partial Y_t^v$. Let $x_t=(\tilde{x}_t,x_{nt}), \ y_t=(\tilde{y}_t,y_{nt})\in Y_t^v$ and take the segment connecting them $\gamma_t(s)=(1-s)x_t+sy_t$. If $Im(\gamma_t)\subset Y_t^v$, the inner and outer distance of $x_t$ and $y_t$ on $Y_t^v$ are equal. If $Im(\gamma_t) \not\subset Y_t^v$, we have a finite decomposition into segments $Im(\gamma_t)=\cup Im(\beta_{i,t})$, where each $\beta_{i,t}$ satisfies the following: $Im(\beta_{i,t})\subset \overline{Y}_t^v $ or $Im(\beta_{i,t})\cap \bar\partial Y_t^v=\{\beta_{i,t}(0), \beta_{i,t}(1)\}$. For each segment $\beta_{i,t}$ satisfying $Im(\beta_{i,t})\cap \overline{Y}_t^v=\beta_{i,t}\cap \bar\partial Y_t^v=\{\beta_{i,t}(0), \beta_{i,t}(1)\}$, we replace that $\beta_{i,t}$ by the arc $(\tilde{\beta}_{i,t},f_j\circ \tilde\beta_{i,t})$ according $\{\beta_{i,t}(0), \beta_{i,t}(1)\} \subset Graph(f_j|_{\tilde Y_t^v})$ for some $j=1,2$. Here, $\tilde{\beta}_{i,t}$ is the arc in $\tilde Y_t^v$ realizing the distance between $\tilde{\beta}_{i,t}(0)$ and $\tilde{\beta}_{i,t}(1)$, where these points are  respectively the projections of $\beta_{i,t}(0)$ and  $\beta_{i,t}(1)$ on the closure of $\tilde Y_t^v$. Since $\|\nabla f_1\|$ and  $\|\nabla f_2\|$ are bounded functions, there is a constant $M>0$ such that the new arc $\tilde\gamma_t$ constructed in this way satisfies $Length(\tilde\gamma_t)\leq M\|x_t-y_t\|$, where the constant $M$ only depends on the LNE constant of $\tilde Y_t^v$ and the maximums of $\|\nabla f_1\|$ and  $\|\nabla f_2\|$. 
\end{proof}
Therefore, $Y$ is LLNE at $0$ w.r.t. $\|\cdot\|_{max,v}$ and by Proposition \ref{LLNE_metric_inv}, $Y$ is LLNE at $0$.
\end{proof}

\begin{remark}\label{closure_L-regular_llne}
Let $Y \subset \mathbb{R}^n$ be an L-regular cell. Then $\overline{Y}$ is LLNE at $p \in \overline{Y}$ whenever its link at $p$, $link_p(\overline{Y})$, is connected.
\end{remark}

Kurdyka in \cite{Kurdyka:1992} proved that any bounded subanalytic set is a finite union of  disjoint $L$-regular cells. This result, Proposition \ref{L-regular_llne} and Remark \ref{closure_L-regular_llne} imply the following result:
\begin{proposition}\label{LLNEdecomposition}
Let $X \subset \mathbb{R}^n$ be a subanalytic set, $0 \in X$. For any bounded subanalytic neighbourhood $\mathcal{U} \subset \mathbb{R}^n$ of $0$, there is a decomposition $X\cap \mathcal{U}=\bigcup_{i=1}^rX_i$, where each $X_i$ and its closure are LLNE at $0$.
\end{proposition}

\begin{definition}
Let $X_1, X_2 \subset \R^n$ be two LLNE sets at $x$ such that $X_1\cap X_2\setminus \{x\}\neq \emptyset$ as a germ at $x$. We say that the pair $(X_1,X_2)$ has \emph{distorted inner-outer link} at $x$ when $X_1\cup X_2$ is not LLNE at $x$.
\end{definition}
\begin{remark}\label{adjacency}
Let $X\subset \R^n$ be a non-empty subset, $x \in \overline{X}$. If $X$ admits a finite LLNE decomposition $\mathcal{X}=\{X_i\}_i$ at $x$ such that each pair $X_k, X_j \in \mathcal{X}$ has no distorted inner-outer link, then $X$ is LLNE at $x$.
\end{remark}

\section{Main Result}\label{sec:main_results}
In this Section, we establish the equivalence of the LNE and LLNE properties and, as a consequence,	 we prove that the LNE property is conical.

%
\begin{theorem}\label{main_thm_connected}
Let $X \subset \mathbb{R}^n$ be a closed subanalytic set, $0 \in X$. Assume that $(X\setminus \{0\},0)$ is a connected germ. Then, $X$ is LNE at $0$ if and only if $X$ is LLNE at $0$.
\end{theorem} 
\begin{proof}
 

It follows immediately from Arc Criterion Theorem (see \cite[Theorem 2.2]{BirbrairM:2017}) that if $X$ is LLNE at $0$ then $X$ is LNE at $0$. Indeed, assume that $X$ is not LNE at $0$. Then, by Arc Criterion Theorem, there is a pair of subanalytic arcs in $X$ passing through $0$ such that $
\|\gamma_1(t)-\gamma_2(t)\|\ll d_{X}(\gamma_1(t),\gamma_2(t))$ and $\|\gamma_1(t)\|=\|\gamma_2(t)\|=t$ for all small enough $t\geq 0$. 
Since $d_{X}(\gamma_1(t),\gamma_2(t))\leq d_{X_t}(\gamma_1(t),\gamma_2(t))$, for all small enough $t>0$, it follows that $\|\gamma_1(t)-\gamma_2(t)\|\ll d_{X_t}(\gamma_1(t),\gamma_2(t))$. So, $X$ is not LLNE at $0$.


Now, we are going to prove that if $X$ is LNE at $0$ then $X$ is LLNE at $0$. Assume that $X$ is LNE at $0$. 

We need a preliminary comment: Given a closed subanalytic set $X$, $0 \in X$, we say that $X$ is LLNE by arcs at $0$, if there is a constant $K\geq 1$ such that for any pair of subanalytic arcs  $\gamma_1, \gamma_2 \colon [0,\epsilon)\to X$ satisfying $\|\gamma_1(t)\|=\|\gamma_2(t)\|=t$ for all $t\in [0,\epsilon)$, we have $d_{X_t}(\gamma_1(t),\gamma_2(t))\leq K\|\gamma_1(t)-\gamma_2(t))\|$ for all small enough $t>0$. 

Since we are assuming that $(X\setminus \{0\},0)$ is a connected germ, we have the following:

\begin{claim}\label{claim_criterio_arc_llne}
$X$ is LLNE at $0$ if and only if $X$ is LLNE by arcs at $0$.
\end{claim}
\begin{proof}[Proof of Claim \ref{claim_criterio_arc_llne}]
If $X$ is LLNE at $0$ then it is immediate to see that $X$ is LLNE by arcs at $0$. 

On the other hand, let us suppose that $X$ is not LLNE at $0$. Then, there are sequences $\{(x_n,y_n)\}_n \subset X \times X$ and $\{t_n\}_n\subset(0,+\infty)$ satisfying $\|x_n\|=\|y_n\|=t_n$ for all $n$,  $(x_n,y_n)\to (0,0)$ and $\frac{\|x_n-y_n\|}{d_{X_{t_n}}(x_n,y_n)}\to 0 $. 
Using Proposition \ref{LLNEdecomposition}, we can assume that $X=\cup_i X_i$ near to $0$, where each $X_i$ is the closure of an $L$-regular cell (which is also LLNE at $0$).
We can assume that $\{x_n\}_n\subset X_l$ and $\{y_n\}_n\subset X_k$ with $X_l\cap X_k\setminus\{0\}\neq \emptyset$ and $d_{X_t}(x_n,y_n)=d_{{(X_l\cup X_k)}_t}(x_n,y_n)$ for all $n$ (see Remark \ref{adjacency}). We consider the subset $Z=\{(x,y,z,t,\epsilon)\in X^3\times \R_+\times \R_+;\|x\|{=}\|y\|{=}\|z\|=t, 0<\frac{\|x-y\|}{\|x-z\|+\|z-y\|}\leq \epsilon\}\subset \R^{3N}\times \R_+\times \R_+$. We denote by $Z_{kl}$ the subset $Z\cap (X_k\times X_l\times [X_k\cap X_l]\times \R^2_+)$. Let $P\colon \R^{3N}\times \R^2\rightarrow \R^{2N}\times \{0\}$ be the linear projection given by $P(x,y,z,t,\epsilon)=(x,y,0,0,0)$. For any $(x,y,0) \in P(Z_{kl})$, we consider the subset $Distortion(Z_{kl})$ formed by the points $(x,y,z,\|x\|,\epsilon) \in Z_{kl}$ with $(x,y,z,\|x\|)$ satisfying    
\[
\|x-z\|+\|z-y\|\leq \|x-w\|+\|w-y\|, \ \forall w, \ (x,y,w,\|x\|,\epsilon) \in Z_{kl}.
\]

It is clear from the construction of $Distortion(Z_{kl})$ the following inequality: 

\begin{equation}\label{well_positioned_one}
\frac{\|x-y\|}{d_{X_{t}}(x,y)}\leq \frac{\|x-y\|}{\|x-z\|+\|z-y\|}, \quad \forall (x,y,z,t,\epsilon) \in Z_{kl}.
\end{equation}

\begin{claim}\label{claim_dist} 
$Distortion(Z_{kl})$ is a non-empty germ at $0$.
\end{claim}

\begin{proof}[Proof of Claim \ref{claim_dist}] 
Let $\{(x_n,y_n)\}_n \subset X_k\times X_l$ and $\{t_n\}_n\subset (0,\infty)$ be the previous sequences and let $\{z_n\}_n$ and $\{\epsilon_n'\}_n$ be sequences such that $\{(x_n,y_n,z_n,t_n,\epsilon_n')\}_n \subset Z_{kl}$. By Proposition \ref{L-regular_llne}, there exists $C\geq 1$ such that $X_k$ and $X_l$ are $C$-LLNE at $0$. Thus, from the facts that $X_k$ and $X_l$ are $C$-LLNE at $0$, $(x_n,z_n) \in X_k\times X_k$ and $(y_n,z_n) \in X_l\times X_l$ for all $n$, we have the following inequality:
\begin{equation}\label{well_positioned_two}
\frac{1}{C}\frac{\|x_n-y_n\|}{\|x_n-z_n\|+\|z_n-y_n\|}\leq \frac{\|x_n-y_n\|}{d_{X_{t_n}}(x_n,y_n)}.
\end{equation}
For each $n$, let $\epsilon_n=\frac{C\|x_n-y_n\|}{d_{X_{t_n}}(x_n,y_n)}$. Since we are assuming that $\frac{\|x_n-y_n\|}{d_{X_{t_n}}(x_n,y_n)}\to 0 $, we obtain that $\{(x_n,y_n,z_n,t_n,\epsilon_n)\}_n \subset Distortion(Z_{kl})$ and $(x_n,y_n,z_n,t_n,\epsilon_n)\to 0$, which implies that $Distortion(Z_{kl})\neq \emptyset$ as a germ at $0$.
\end{proof}
Now, notice that $Distortion(Z_{kl})$ is a subanalytic germ at $0$. Since we are assuming that $X$ is not LLNE at $0$, by Claim \ref{claim_dist}, $Distortion(Z_{kl})$ is non-empty as a germ at $0$ and, thus, we can take subanalytic arcs $\gamma_1,\gamma_2,\gamma_3\colon[0,\epsilon)\to X$ and $\eta \colon (0,\epsilon) \to \R_+$ such that $(\gamma_1(t),\gamma_2(t),\gamma_3(t),t,\eta(t))\in Distortion(Z_{kl})$ for all $t\in (0,\epsilon)$ and $\eta(t)\to 0$ when $t\to 0$. It follows from Inequality (\ref{well_positioned_one}) that $\frac{\|\gamma_1(t)-\gamma_2(t)\|}{d_{X_t}(\gamma_1(t),\gamma_2(t))}\to 0$, when $t\to 0$. Therefore $X$ is not LLNE by arcs at $0$, which finishes the proof of Claim \ref{claim_criterio_arc_llne}.
\end{proof}

Suppose by contradiction that $X$ is not LLNE at $0$. By Claim \ref{claim_criterio_arc_llne}, there is a pair of arcs $\gamma_1,\gamma_2\colon[0,\epsilon)\to X$ such that $\|\gamma_j(t)\|=t, \ j=1,2$ and $\frac{\|\gamma_1(t)-\gamma_2(t)\|}{d_{X_t}(\gamma_1(t),\gamma_2(t))}\to 0$, when $t\to 0$. From \cite{BirbrairM:2000} (see also \cite{Kurdyka:1997}), we can construct a metric $d_P$ on $X$ (bi-Lipschitz equivalent to $d_X$), considering the LLNE decomposition $X=\cup_j X_j$. By definition of the metric $d_P$, we can choose a finite number of subanalytic arcs $\tilde \beta_1, \ldots \tilde \beta_r, \ \tilde{\beta}_i(0)=0, \ i=1,\ldots,r$ such that
\begin{eqnarray*}
d(t)&:=&d_P(\gamma_1(t),\gamma_2(t))\\
    &\simeq& \|\gamma_1(t)-\tilde{\beta}_1(t)\|+\|\tilde \beta_1(t)-\tilde{\beta}_2(t)\|+\ldots+\|\tilde \beta_r(t)-\gamma_2(t)\|=:\tilde d(t).
\end{eqnarray*}
We can assume that the image of each pair $\tilde \beta_i$ and  $\tilde \beta_{i+1}$ is contained in some $X_j$. We can see that there is a minimal $s\in\{0,1,\ldots,r\}$ such that $tord_X(\gamma_1,\gamma_2)=ord_t\|\tilde \beta_s(t)-\tilde \beta_{s+1}(t)\|$, where $\tilde \beta_0=\gamma_1$ and $\tilde \beta_{r+1}=\gamma_2$. These arcs $\tilde\beta_i$'s defining the function $\tilde d$ are not necessarily parametrized by the distance to $0$, but using the inner and outer Order Comparison Lemmas (see \cite[Lemma 2.5]{BirbrairM:2017} and \cite[Order Comparison Lemma]{BirbrairF:2000}), we obtain that $ord_t\|\tilde \beta_s(t)-\tilde \beta_{s+1}(t)\|=tord(\tilde \beta_s,\tilde \beta_{s+1})$. Therefore $tord_X(\gamma_1,\gamma_2)=tord(\tilde \beta_s, \tilde \beta_{s+1})$.
For each $i\in\{0,1,\ldots,r+1\}$, let $\beta_i$ be the parametrization of $\tilde \beta_i$ by the distance to the origin and let $h\colon [0,\epsilon)\to \R$ be the function given by 

\vspace{0.25cm}
$h(t)=d_{X_t}(\gamma_1(t),\beta_1(t))+d_{X_t}(\beta_1(t),\beta_2(t))+ \ldots +d_{X_t}(\beta_r(t),\gamma_2(t))$.

\vspace{0.25cm}
Since $d_{X_t}(\gamma_1(t),\gamma_2(t)) \leq h(t)$ for all $t$, then $\|\gamma_1(t)-\gamma_2(t)\|\ll h(t)$.
Notice that each term $d_{X_{t}}( \beta_i(t), \beta_{i+1}(t))$ is positive for $t\in (0,\epsilon)$. Then, there exists $i\in\{0,1,\ldots,r\}$ such that $ \|\gamma_1(t)-\gamma_2(t)\|\ll d_{X_t}(\beta_i(t),\beta_{i+1}(t))$.
On the other hand, we are assuming that $X$ is LNE at $0$. Then $tord_X(\gamma_1,\gamma_2)=tord(\tilde \beta_s, \tilde \beta_{s+1})=tord(\beta_s, \beta_{s+1})=tord(\gamma_1,\gamma_2)$. Hence,
$\|\beta_s(t)-\beta_{s+1}(t)\|\ll d_{X_t}(\beta_i(t), \beta_{i+1}(t)) $. By the description of the function $\tilde d$ above, we must have that $tord(\beta_i, \beta_{i+1})\geq tord(\beta_s, \beta_{s+1})$. Hence, this inequality implies that $\| \beta_i(t)- \beta_{i+1}(t)\|\ll d_{X_t}(\beta_i(t), \beta_{i+1}(t)) $, where $ Im(\beta_i), Im( \beta_{i+1}) \subset X_j$ for some $j$, which is a contradiction, since $X_j$ is LLNE at $0$. Therefore $X$ is LLNE at $0$, which finishes the proof.

\end{proof}

\begin{corollary}
Let $X \subset \mathbb{R}^n$ be a closed subanalytic set, $0 \in X$. Assume that $(X\setminus \{0\},0)$ is a connected germ. Then, $X$ is LNE at $0$ if and only if $X$ is LLNE at $0$ w.r.t. any subanalytic norm on $\R^n$.
\end{corollary}
As a direct consequence, we obtain the following version of Theorem \ref{main_thm_connected}, which works even when $(X\setminus \{0\},0)$ is not a connected germ (see definition \ref{dist.connect.comp}).
\begin{corollary}\label{main_thm_non-connected}
Let $X \subset \mathbb{R}^N$ be a closed subanalytic set, $0 \in X$. Let $C_1,...,C_r$ be the connected components of $X\setminus \{0\}$ (as a germ at $0$). Then, the following statements are equivalent:
\begin{itemize}
\item [(i)] $X$ is LNE at $0$;
\item [(ii)] Each $\overline{C}_i$ is LNE at $0$ and there exists $K>0$ such that $d_0(X_t)\geq Kt$ for all small enough $t>0$;
\item [(iii)] Each $\overline{C}_i$ is LLNE at $0$ and there exists $K>0$ such that $d_0(X_t)\geq Kt$ for all small enough $t>0$.
\end{itemize}
\end{corollary}

Let us remark that the previous theorem does not hold true if we remove the subanalytic condition.

\begin{example}\label{ex:non-LNE}
Let $f\colon [0,1]\to \mathbb{R}^2$ be the function given by 
$$
f(x)=\left\{\begin{array}{ll}
xe^{\frac{2\pi i}{x}},&\, x\not=0\\
0,&\, x=0.
\end{array}\right.
$$
Then $X=Im(f):=\{f(x);x\in [0,1]\}$ is not LNE at $0\in \mathbb{R}^2$ (see Figure \ref{espiral}) but $X_t:=X\cap \mathbb{S}_t^2=\{te^{\frac{2\pi i}{t}}\}$ is LNE for all $t>0$ and, moreover, $d_{X_t}=\|\cdot\|$.
\end{example}

\begin{figure}[H]
    \centering \includegraphics[scale=0.41]{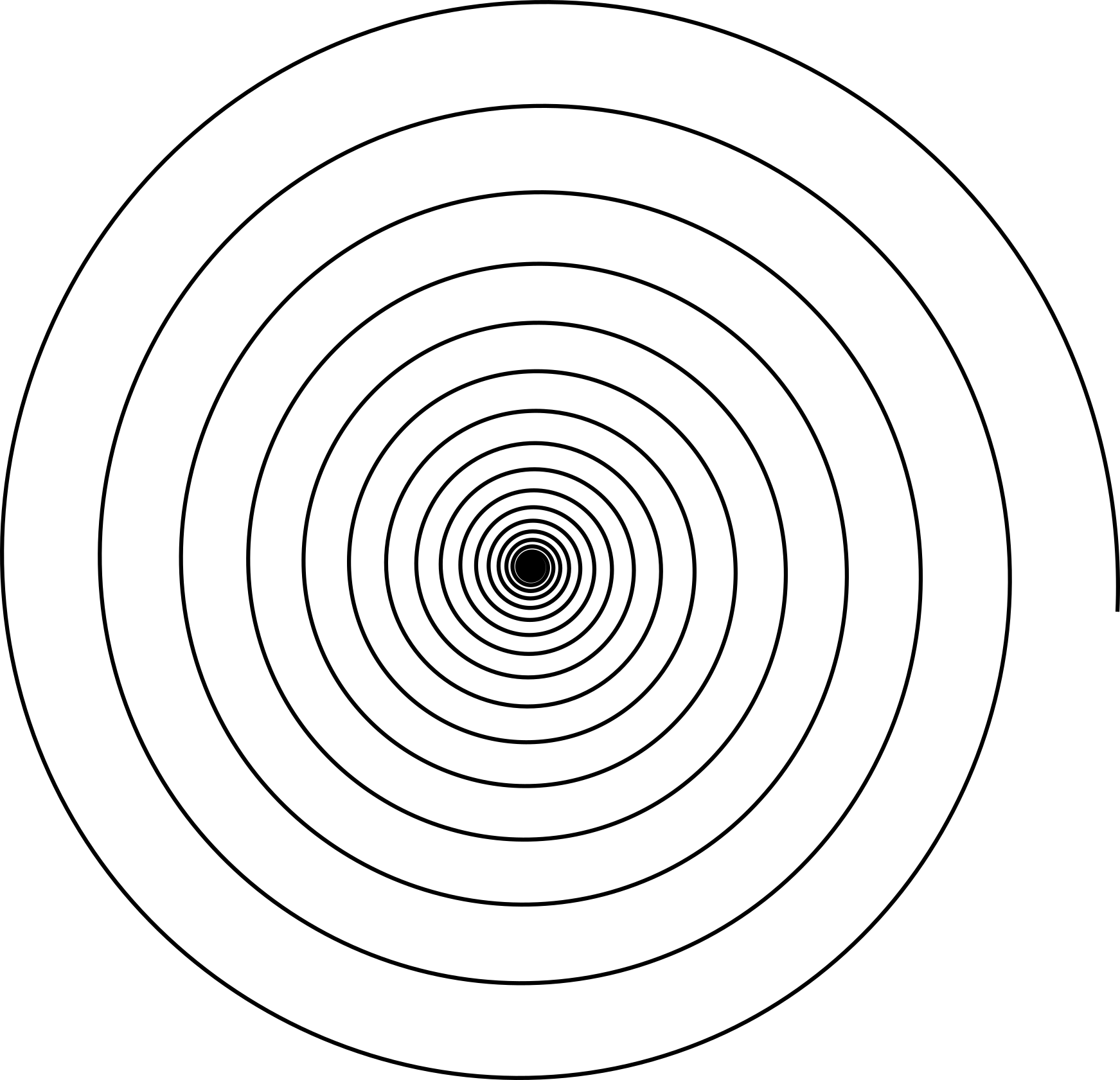}
  \caption{Spiral $X=Im(f)$}\label{espiral}
\end{figure}

\begin{example}\label{ex:non-LLNE}
For each positive integer $j$, we consider $Y_j=\{(x,y,z)\in \R^3;\, x^2+(z-1/j)^2=y^3 \mbox{ and } 0\leq y\leq \big(\frac{1}{4j(j-1)}\big)^\frac{2}{3}\}.$ Let $X=\bigcup\limits_{j=0}^\infty Y_j$, where $Y_0=\{(x,y,z)\in \R^3;\, x=y=0\}$ (see Figure \ref{several_horns}). Then, $X$ is LNE at $0$. However, $X\cap \mathbb{S}_{1/j}^2$ is not LNE for all positive integer $j$, which implies that $X$ is not LLNE at $0$.
\end{example}
\begin{figure}[H]
    \centering \includegraphics[scale=0.8]{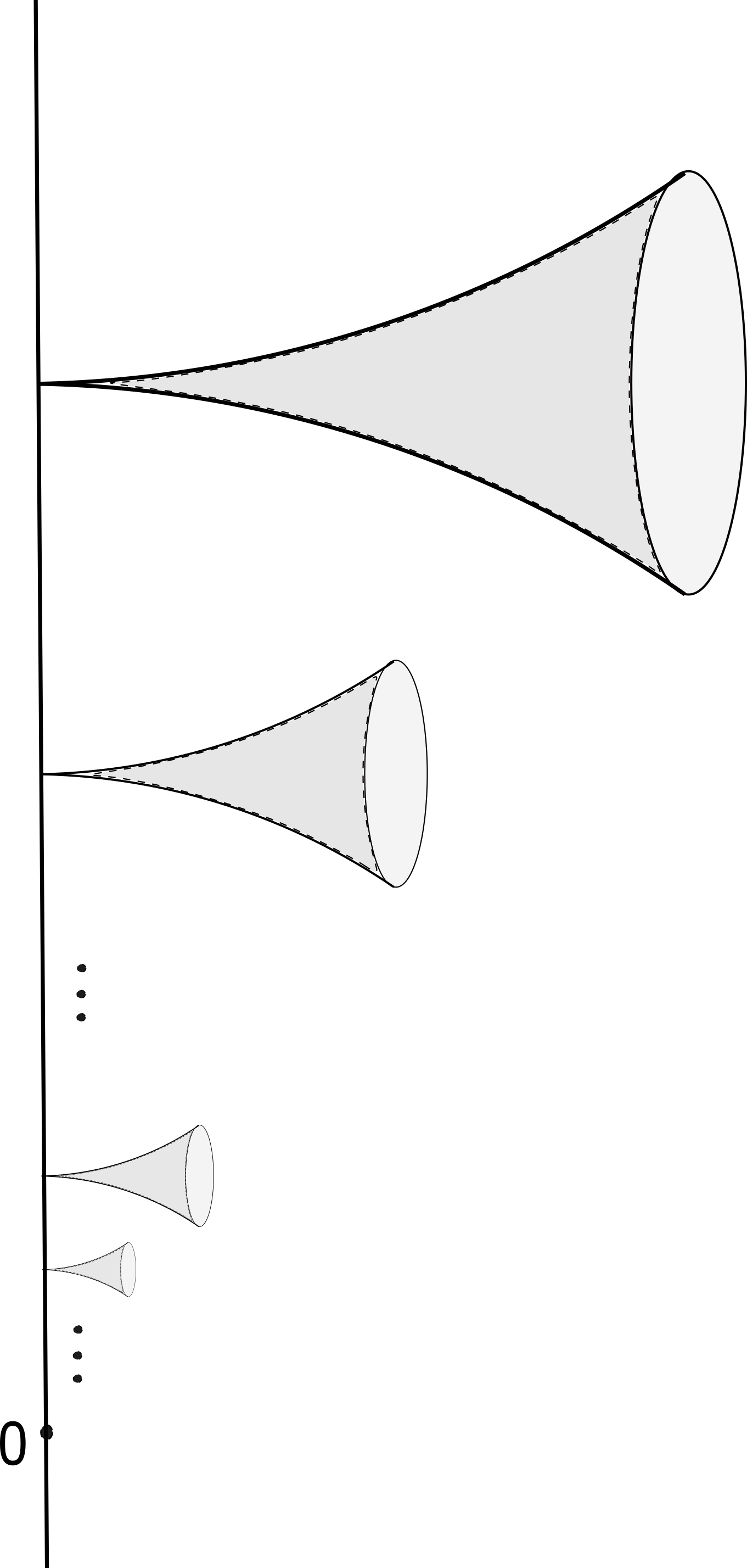}
  \caption{Infinitely many horns and a line}\label{several_horns}
\end{figure}


\end{document}